\documentclass[letterpaper, 10 pt, conference]{ieeeconf}
\usepackage{graphicx}

\usepackage{amsmath}
\usepackage{amsfonts}

\newtheorem{theorem}{Theorem}

\newtheorem{remark}{Remark}

\newcommand{\pder}[2][]{\frac{\partial#1}{\partial#2}}

\newcommand{\RR}{\mathbb{R}}

\begin{document}

\title{\LARGE \bf Control Contraction Metrics, Robust Control, and Observer Duality}

\IEEEoverridecommandlockouts
\author{
Ian R. Manchester$^1$ \ \ Jean-Jacques E. Slotine$^2$ \\ \ \\
1: ACFR, School of Aerospace, Mechanical and Mechatronic Engineering, University of Sydney, Australia\\
2: Nonlinear Systems Laboratory, Massachusetts Institute of Technology, USA\\
ian.manchester@sydney.edu.au \ jjs@mit.edu}
\maketitle

\begin{abstract}
This paper addresses the problems of stabilization, robust control, and observer design for nonlinear systems. We build upon recently a proposed method based on contraction theory and convex optimization, extending the class of systems to which it is applicable. We prove converse results for mechanical systems and feedback-linearizable systems. Next we consider robust control, and give a simple construction of a controller guaranteeing an $L^2$-gain condition, and discuss connections to nonlinear $H^\infty$ control. Finally, we discuss a ``duality'' result between nonlinear stabilization problems and observer construction, in the process constructing globally stable reduced-order observers for a class of nonlinear systems.
\end{abstract}

\section{Introduction}


Constructive control design for nonlinear systems remains a challenging problem. The classical Lyapunov stability theory leads to necessary and sufficient conditions in terms of existence of control Lyapunov functions \cite{sontag1983lyapunov}, \cite{artstein1983stabilization}, however these may be difficult to find \cite{rantzer2001dual}. Constructive methods, such as feedback linearization \cite{isidori1995nonlinear}, backstepping \cite{krstic1995nonlinear}, energy-based methods  \cite{schaft1999l2}, and sliding or adaptive control \cite{slotine1991applied} are generally applicable only to a limited class of systems.

It is common to pose the search for a feedback control law and an associated performance certificate as an optimization problem. There has been significant work over the last two decades on finding convex representations for such problems, e.g. using linear matrix inequalities and sum-of-squares programming \cite{boyd1994linear}, \cite{parrilo2003semidefinite}. For nonlinear control design, the density functions of \cite{rantzer2001dual}, \cite{prajna2004nonlinear} and related techniques of occupation measures \cite{lasserre2008nonlinear} and control Lyapunov measures \cite{vaidya2010nonlinear} explicitly address convexity of criteria. Another approach is to piece together locally stabilized trajectories, with regions of stability verified via sum-of-squares programming \cite{tedrake2010lqr}.

An alternative to searching for an explicit control law is to embrace real-time optimization in the feedback loop, as in nonlinear model predictive control (NMPC) \cite{diehl2009efficient}. NMPC has many attractive qualities, especially in the handling of constraints. However, for high-speed nonlinear systems or low-cost hardware the computation time may still be prohibitive, and it is difficult to analyse performance and robustness of nonlinear MPC schemes.

In \cite{manchester2014control}, and continued in this paper, we propose a method that could be considered a middle-ground between these two approaches. There is an off-line stage, during which a nonlinear convex optimization problem is solved for a {\em control contraction metric} (CCM). This metric defines a distance between points and trajectories which can be thought of as an infinite family of control Lyapunov functions. There is also an on-line phase, in which a minimal path with respect to this metric is computed, and a differential control law is integrated along this path. 

Compared to an explicit control law, this strategy demands more on-line computation, however the search for a CCM can be made convex in a way that naturally extends the linear theory. Compared to NMPC, it has more limited applicability and does not directly address hard constraints, however the on-line component is computationally simpler than a nonlinear MPC problem, and it is straightforward to ensure robust stability.



The idea of a CCM extends contraction analysis \cite{Lohmiller98} to control synthesis. Contraction analysis is based on the differential dynamics, global stability results are derived from local criteria, and the problem of motion stability is decoupled from the choice of a particular solution. For polynomial systems, the search for a contraction metric can be formulated as a convex optimisation problem using sum-of-squares programming  \cite{aylward2008stability}.

Historically, basic convergence results on contracting systems can be traced back to the results of \cite{lewis1949metric} in terms of Finsler metrics. Weaker forms of contraction allow the study of limit cycles and the dimension of chaotic systems  \cite{manchester2014transverse}, \cite{boichenko}. Contraction is closely related to incremental stability, i.e. the stability of pairs of solutions (see, e.g., \cite{desoer1975feedback} and \cite{angeli2002lyapunov}) and the related concept of convergent dynamics \cite{demidovich1962dissipativity}, \cite{ruffer2013convergent}.

A contraction metric can be thought of as a Riemannian metric with the additional property that differential displacements get smaller (with respect to the metric) under the flow of the system. A control contraction metric has the property that differential displacements can be {\em made} to get shorter by control action. This is analogous to the relationship between a Lyapunov function and a control Lyapunov function.

The conditions derived in \cite{manchester2014control} for stabilizability take the form of state-dependent linear matrix inequalities. The conditions are superficially similar to those for state-feedback synthesis via global linearization \cite[Ch. 7]{boyd1994linear}, and the closely related results for output regulation using convergence theory \cite{pavlov2006uniform}. However, the control construction we propose allows the use of non-quadratic (state-dependent) metrics without requiring difficult integrability conditions on gains. There are also similarities to LPV control synthesis and gain scheduling (e.g. \cite{apkarian1995convex}) however the class of systems considered is different.

In this paper, we build upon the results of \cite{manchester2014control} in a number of ways. Firstly, in Section \ref{sec:CCM} we extend the class of systems to which the method of \cite{manchester2014control} is applicable. In Section \ref{sec:converse} we give two converse results. In Section \ref{sec:robust} we extend the CCM method to a robust control problem similar to nonlinear $H^\infty$. Finally, in Section \ref{sec:observer} we discuss ``dual'' conditions for observers, and provide a novel procedure for reduced-order observer design.



\section{Preliminaries}

We consider a nonlinear time-dependent control-affine system
\begin{equation}\label{eq:sys}
\dot x = f(x,t) +B(x,t)u
\end{equation}
where $x(t)\in\RR^n, u(t)\in\RR^m$ are state and control, respectively, at time $t\in\RR^+:=[0,\infty)$. The functions $f: \RR^n\times \RR^+ \rightarrow \RR^n$ and $B:\RR^n\times \RR^+\rightarrow \RR^{n\times m}$ are assumed smooth. Note that this is a wider class of systems than considered in \cite{manchester2014control}, where it was assumed that $B$ was independent of $x$.

Contraction analysis is based on the study of an {\em extended system} consisting of \eqref{eq:sys} and its associated system of differential (a.k.a. variatonal, linearized) dynamics:
\begin{equation}\label{eq:diffdyn}
\dot\delta_x(t) = A(x,u,t)\delta_x(t)+B(x,t)\delta_u(t)
\end{equation}
where $A(x,u,t) = \pder{x}(f(x,t)+B(x,t)u)$ is affine in $u$.

As is standard, a solution $(x^\star, u^\star)$ defined on $[0, \infty)$ is said to be globally asymptotically stabilized by a feedback controller $u=k(x,t)$ if a closed-loop solution $x(t)$ exists on $t\in [0, \infty)$ and satisfies
\begin{enumerate}
\item For any $\alpha$ there exists an $\epsilon$ such that $|x_0-x^\star(0)|<\epsilon$ implies $|x(t)-x^\star(t)|<\alpha$,
\item For any initial condition $x_0\in \RR^n$, the closed loop solution satisfies $|x(t)-x^\star(t)|\rightarrow 0$.
\end{enumerate}
Global exponential stabilization refers to the stronger condition that there exists a $K$ and $\lambda$ such that \[|x(t)-x^\star(t)|\le K e^{-\lambda t}|x_0-x^\star(0)|\] for all $x(0)$.

Following \cite{manchester2014control}, a system of the form \eqref{eq:sys} is said to be {\em universally stabilizable by state feedback} if for every solution $(x^\star, u^\star)$ defined on $t\in [0, \infty)$ there exists a state feedback controller $k:\RR^n\rightarrow \RR^m$ such that $(x^\star, u^\star)$ is globally stabilized by $u=k(x,t)$. Analogously, a system can be {\em universally exponentially stabilizable with rate $\lambda$}.

Note that universal stabilizability is a stronger condition than global stabilizability of a particular solution (e.g. the origin).

%
%
%

\subsection{Riemannian Metrics}
Here we briefly recall some relevant facts from Riemannian geometry \cite{docarmo1992riemannian}.
Riemannian geometry provides a way of extending intuitive notions of Euclidean geometry to the study of distances and curvature on more general nonlinear manifolds. In this paper, the underlying state space remains $\RR^n$, but we use Riemannian metrics as generalised distances between points for the purpose of motion stabilization.

A Riemannian metric is a symmetric positive-definite matrix function $M(x)$, smooth in $x$, which defines a ``local Euclidean'' structure on a manifold, by way of the inner product $\langle\delta_1, \delta_2\rangle_x = \delta_1'M(x)\delta_2$ for any two tangent vectors $\delta_1, \delta_2$, and the norm $\sqrt{\langle\delta, \delta\rangle_x}$.

Let $\Gamma(a,b)$ be the set of smooth paths between two points $a$ and $b$, where each $\gamma\in\ \Gamma(a,b)$ is a smooth mapping $\gamma:[0, 1] \rightarrow \RR^n$ and satisfying $\gamma(0) = a$ and $\gamma(1) = b$. We use the notation $\gamma(s), s\in[0, 1]$ and $\gamma_s(s) := \pder[\gamma]{s}$. Given a metric $M(x)$, we can define the path {\em length}:
\[
l(\gamma) := \int_0^1 \sqrt{\langle \gamma_s, \gamma_s\rangle_s} ds 
\]
and {\em energy}, i.e. integral of squared-length:
\[
e(\gamma) := \int_0^1 \langle \gamma_s, \gamma_s\rangle_s ds.
\]
The Riemannian distance $d(a,b)$ between two points is the length of the shortest path $\gamma$ between them, and satisfies the properties of the metric. In our context the Hopf-Rinow theorem implies the existence of a minimal path, which is smooth and a geodesic. We also use the notation $e(a,b)=e(\gamma)$ for the energy of the minimal curve. Given the minimal curve $\gamma$ and an arbitrary curve $c \in\Gamma(a,b)$, we have:
\[
e(\gamma) = l(\gamma)^2\le l(c)^2 \le e(c).
\]
This relation is useful because it implies that paths of minimal energy and distance are the same, and in many respects the energy function is more convenient due to its smoothness. We note that the distance function is invariant to reparameterization of $\gamma$, while the energy function is not.


\section{Control Contraction Metrics}\label{sec:CCM}

We now give the basic idea of a control contraction metric (CCM), proposed by the authors in \cite{manchester2014control}. Suppose a system has the property that every solution is locally stabilizable. Each local stabilization may have small region of stability, but if a ``chain'' of states joining the current state $x$ to $x^\star(t)$ is stabilized, in the sense that if each ``link'' in the chain gets shorter, then $x(t)$ is driven towards $x^\star(t)$.

Construction of a CCM is based on taking this concept to the limit as the number of links in the chain goes to infinity, and becomes a smooth path $\gamma(s)$ connecting $x^\star(t)$ and $x(t)$ in the state space. The differential dynamics \eqref{eq:diffdyn} describe the dynamics of infinitesimal path segments. Now, suppose one can find a Riemannian metric $\langle \delta_x, \delta_x\rangle_x=\delta_x'M(x,t)\delta_x$ which verifies that a differential feedback law $\delta_u = K(x,u,t)\delta_x$, affine in $u$, is stabilising, i.e.
\begin{equation}\label{eq:ccm1}
\dot M+(A+BK)'M + M(A+BK)<0, \quad \forall x,u,t
\end{equation}
then we refer to $M(x,t)$ as a control contraction metric. The $<0$ above can be replaced with $\le -2\lambda M$ for exponential stability with rate $\lambda$. The control signal applied is then computed by integrating the differential control signals $\delta_u$ along the path $\gamma$, i.e. solving
\begin{equation}\label{eq:control_int}
u(t,s) = u^\star(t)+\int_0^s K(\gamma(t,s),u(t,s),t) \gamma_s(t,s)ds,
\end{equation}
over the interval $s\in[0,1]$ and applying the control signal $u(t) = u(t,1)$.
Note that, since $K$ is affine in $u$, the above integral equation is guaranteed to have a unique solution. If $K$ is independent of $u$ then the integral equation reduces to quadrature.

\begin{theorem} Suppose a system \eqref{eq:sys}, \eqref{eq:diffdyn} satisfies \eqref{eq:ccm1} for some $M(x,t)$ satisfying uniform bounds $\alpha_1I\le M(x,t)\le \alpha_2 I$ with $\alpha_1>0$, then the system is universally stabilizable by state feedback, in particular the control law \eqref{eq:control_int}. 
When \eqref{eq:ccm1} has right-hand-side $-2\lambda M$ the system is universally exponentially stabilizable.
\end{theorem}

We omit the proof as it is similar to \cite{manchester2014control}.

Note that the controller \eqref{eq:control_int} does {\em not} necessarily make the resulting closed-loop system contracting in every direction, only along the minimal path $\gamma$. This is in contrast to the methods of \cite{pavlov2006uniform} and avoids the problems of integrability discussed in \cite{lu1997robustness} and \cite{fromion2003theoretical}.

It is well-known that Condition \eqref{eq:ccm1} is equivalent to the existence of $W(x,t)>0$, $Y(x,u,t)$, with the latter affine in $u$, satisfying:
\begin{equation}\label{eq:CCM_Yform}
-\dot W + AW +WA' + BY + Y'B'<0
\end{equation}
by taking $W=M^{-1}$ and $K = YW^{-1}$. Another useful condition is the existence of $W(x,t)>0$ and a scalar function $\rho(x,u,t)$, affine in $u$,
\begin{equation}\label{eq:CCM_rhoform}
-\dot W + AW +WA' -\rho BB'<0.
\end{equation}
which implies \eqref{eq:CCM_Yform} by taking $Y = -\frac{1}{2}\rho B'$. Each of \eqref{eq:CCM_Yform}, \eqref{eq:CCM_rhoform} clearly implies
\begin{equation}\label{eq:CCM_kernelform}
\delta(-\dot W + AW +WA')\delta <0, \quad \forall \delta:B'\delta = 0.
\end{equation}
In each case, the $<0$ can be replaced by $\le-2\lambda W$ for exponential stability with rate $\lambda$.

It is always the case that \eqref{eq:CCM_rhoform} $\Rightarrow$ \eqref{eq:CCM_Yform} $\Rightarrow$ \eqref{eq:CCM_kernelform}. In the case when $\rho(x,u,t)$ is an unrestricted function of $(x,u,t)$, Finsler's theorem states that  \eqref{eq:CCM_kernelform}$\Rightarrow$\eqref{eq:CCM_rhoform}, so all three are equivalent. However, we choose to restrict to affine dependence on $u$ to guarantee solvability of \eqref{eq:control_int}.

\begin{remark} Each of the above conditions is ``intrinsic" in the sense that it is invariant under a smooth change of coordinates, e.g. taking \eqref{eq:CCM_Yform} and the transformation $\delta_x\mapsto T(x)\delta_x$, we have the corresponding transformations $A\mapsto \dot TT^{-1}+TAT^{-1}$, $B\mapsto TB$, $W\mapsto TWT'$ and $Y\mapsto YT'$. Straightforward calculations show that the inequality \eqref{eq:CCM_Yform} is preserved, and the resulting geodesics and feedback control are identical.
This fact may be useful in extending our results to manifolds beyond $\RR^n$ by making use of an atlas of local charts \cite{docarmo1992riemannian}.
\end{remark}

\subsection{CLFs and Alternative Control Laws}

It is not essential that the path-integral controller \eqref{eq:control_int} used in the proof is actually implemented on-line. In fact, given a control contraction metric, the distance or energy function between two $x$ and $x^\star$ can be used as a control Lyapunov function for {\em any} feasible target trajectory $x^\star(t)$.

From the formula for first variation of energy \cite{docarmo1992riemannian}, we have the following for the derivative of the energy functional:
\begin{align}
\frac{1}{2}\frac{d}{dt} e(x^\star(t), x(t)) =&\langle \gamma_s(0,t), \dot x^\star(t)\rangle_{x^\star} - \langle \gamma_s(1,t), f(x,t)\rangle_x\notag\\
&-\langle \gamma_s(1,t),B(x,t)u\rangle_x\label{eq:control_cone}
\end{align}
We note that the energy is a smooth function of $x$ and $x^\star$ as long as there is a unique minimal geodesic joining them. If not, the above formula still holds with $\frac{d}{dt}$ interpreted as the right-derivative, and the infimum of the right-hand side is taken over minimal geodesics joining $x$ and $x^\star$.

The control contraction metric condition then implies the CLF-like condition that if $\gamma_s(1,t)'M(x,)B(x,t)u>0$ then $\langle \gamma_s(0,t), \dot x^\star\rangle_{x^\star} - \langle \gamma_s(1,t), f(x,t)\rangle_x<0$.

One can then treat the energy like a CLF and choose any $u$ for which \eqref{eq:control_cone} is negative. This makes precise the intuitive notion that the controller should push the state towards $x^\star$, where ``towards'' is defined by the geodesic and the metric. This would allow, e.g., the use of linear programming to find a stabilizing control signal optimizes some other criterion or bounds.

The formula \eqref{eq:control_cone} also makes clear that any resulting controller has an infinite up-side gain margin, and also gives an indication of how accurately one must compute the geodesic $\gamma$. A large control input can be stabilizing as long as the direction of the vector $B(x,t)'M(x,t)\gamma_s(1,t)$ is known in the $u$ space to within 90 degrees. This may prove useful in reducing the computational complexity of the on-line component of the CCM control strategy.

\section{Converse Results} \label{sec:converse}

The CCM conditions we give are sufficient and not, in general, expected to be necessary for stabilization. Therefore it is interesting to note classes of systems where in fact they are necessary. He we give two such examples: feedback linearizable systems and mechanical systems.

\subsection{Feedback Linearizable Systems}

We say that a time-invariant system of the form
\[
\dot x = f(x,t)+B(x,t)u
\]
is feedback linearizable if there exists a time-varying smooth global diffeomorphism $z = \theta(x,t)$ and a smooth feedback control $\bar u(x,v,t) = \alpha(x,t) + \beta(x,t)v$, with $\beta(x,t)$ nonsingular for all $x,t$, such that transformed system is LTI:
\begin{equation}\label{eq:fl_zsys}
\dot z = Gz+Hv
\end{equation}
where the pair of constant matrices $(G, H)$ is stabilizable. 

When such functions $\theta, \alpha, \beta$ can be found, the feedback stabilization problem is rendered trivial, but a major challenge is that even if such functions can be proven to exist, they may be difficult to find: usually this involves solving a partial differential equation (see, e.g.,  \cite{isidori1995nonlinear}). 

Note that in much of the literature on feedback linearization, it is required that $(G,H)$ be controllable, but our definition obviously covers this case.

\ 

\begin{theorem} For any feedback linearizable system there is a control contraction metric that verifies universal stabilizability, given by $M(x) = \pder[\theta]{x}'X\pder[\theta]{x}$ where $P$ is any positive definite matrix $X$ satisfying, for some gain $L$,
\begin{equation}\label{eq:fblin_lyap}
XG+G'X+XHL+L'H'X<0.
\end{equation}
\end{theorem}

\begin{proof}
We first note that $X=X'>0,L$ satisfying \eqref{eq:fblin_lyap} are guaranteed to exist by assumption that \eqref{eq:fl_zsys} is stabilizable.

The differential dynamics of a linear system \eqref{eq:fl_zsys} are just the dynamics of the system itself:
\[
\dot \delta_z = G\delta_z + H\delta_v.
\]
The global coordinate transformation $z = \theta(x,t)$ admits variation $\delta_z = \Theta(x,t)\delta_x$ where $\Theta:=\pder[\theta]{x}$. Therefore we also have differential coordinates $\delta_z = \Theta(x,t)\delta_x$, so
\begin{align}
\dot \delta_z =&[\dot\Theta \Theta^{-1} + \Theta (A+B\tfrac{\partial {\bar u}}{\partial x}) \Theta^{-1}]\delta_z+  \Theta B \beta\delta_v,\notag
\end{align}
so it must be the case that 
\begin{align}
G &=  [\dot\Theta \Theta^{-1} + \Theta (A+B\tfrac{\partial {\bar u}}{\partial x}) \Theta^{-1}], \quad H = \Theta B \beta.\notag
\end{align}
Now, take \eqref{eq:fblin_lyap} and substitute the above formulas for $G, H$ and we get
\begin{align}
X(\dot\Theta\Theta^{-1}+\Theta A\Theta^{-1})+(\dot\Theta\Theta^{-1}+\Theta A\Theta^{-1})'X& \notag\\
+X\Theta B(\tfrac{\partial {\bar u}}{\partial x} \Theta^{-1}+\beta L) +(\tfrac{\partial {\bar u}}{\partial x} \Theta^{-1}+\beta L)'B'\Theta' X&<0\notag
\end{align}
Multiply on the left by $\Theta'$ and on the right by $\Theta$, and replace $M(x) = \Theta(x)'X\Theta(x)$ and $\dot M(x,u) = \dot\Theta(x,u)'X\Theta(x)+\Theta(x)'X\dot\Theta(x,u)$ and we have the control contraction condition
\begin{align}
\dot M + A'M +MA + MBK + K'B'M<0. \notag
\end{align}
where the gain $K$ is
\[
K(x,u) = \pder[u]{x} = \pder[\bar u]{x} +\pder[\bar u]{v}\pder[v]{x} = \pder[\bar u]{x}+\beta(x)L\Theta(x).
\]
This completes the proof of the theorem.
\end{proof}

\subsection{Mechanical Systems}

For many physical systems, energy is a natural candidate for a control Lyapunov function, and stabilization can be achieved by passivity and damping assignment methods. In fact, for such systems there is also a ``natural'' control contraction metric.

Here we make use of a general fact, straightforward to prove, that if there is a feedback control of the form $u=k(x)+\beta(x)v$ with $\beta(x)$ square and nonsingular, and $v$ an independent control input, that makes a system contracting with respect to some metric $M(x)$, then $M(x)$ is a control contraction metric.

It is well-known that the Euler Lagrange equations for standard-form mechanical system can be
\begin{equation}\label{eq:mechsys}
H(q)\ddot q + (C(q,\dot q)+D(q,\dot q))\dot q + G(q) = u
\end{equation}
where $H(q)>0$ for all $q$ and $\dot H-2C$ is skew-symmetric, damping $D(q,\dot q)\ge 0$, and $u$ is an external torque input.

\begin{theorem} For the class of systems \eqref{eq:mechsys}, and for any $K_P=K_P'>0$, $K_D=K_D'\ge 0$ let
\begin{equation}\label{eq:mech_metric}
\langle \delta_x, \delta_x\rangle :=\frac{1}{2}\begin{bmatrix}\delta_{\dot q}\\ \delta_q\end{bmatrix}'M(q,\dot q)\begin{bmatrix}\delta_{\dot q}\\ \delta_q\end{bmatrix} 
\end{equation}
where, $R = C(q,\dot q)+D(q, \dot q)+K_D$ and
\[
M=\begin{bmatrix} HK_PH & HK_PR\\ R'K_PH & H+R'K_PR\end{bmatrix}.
\]
This defines a control contraction metric, which can be made contracting using a type of PD control:
\[
u(t) = -K_D\dot q(t) -K_P(q(t)-q_0(t))+G(q).
\]
\end{theorem}
We omit the proof as it is similar to the proof of Theorem 2 in \cite{lohmiller1997applications}. Here $q_0(t)$ is considered the independent input.

\begin{remark}
If $K_P=0$ then the control contraction metric given above reduces to $\frac{1}{2}\delta_q'H(q)\delta_q$, i.e. the Riemannian metric associated with kinetic energy.
\end{remark}

\section{Differential $L^2$-Gain and Robust Control}\label{sec:robust}

The small-gain theory and $H^\infty$ control are cornerstones of rigorous analysis and control of uncertain feedback systems. Extensions of $H^\infty$ to nonlinear systems have been widely studied and typically make use of a time-domain $L^2$-gain formulation (see, e.g., \cite{schaft1999l2}, \cite{helton1999extending}).

To begin, consider a state-space system driven by disturbance $w$, with output $y$:
\begin{equation}\label{eq:sysw}
\dot x = f(x,w,t), y = g(x,w,t).
\end{equation}
Assuming that $\dot x = x=y=w=0$ is a solution, we say that the system has $L^2$ gain of $\alpha$ if
\begin{equation}\label{eq:L2}
\int_0^T|y|^2dt \le \alpha^2\int_0^T|w|^2dt
\end{equation}
for all solutions with $T>0$, $w\in L^2[0,T]$, and $x(0)=0$. The standard way to verify \eqref{eq:L2} is to search for a storage function $V$ satisfying the integral constraint:
\[
V(x(T))-V(x(0)) \le \int_0^T[\alpha^2|w|^2-|y|^2]dt
\]
or its differential form $\dot V  \le \alpha^2|w|^2-|y|^2$, which is clearly convex (linear) in $V$.

The search for a {\em feedback controller} guaranteeing that the closed-loop system satisfies an $L^2$ gain involves the challenging problem of solving a Hamilton-Jacobi-Isaacs partial differential equation  \cite{schaft1999l2}, \cite{helton1999extending}. We also note the mixed pointwise-LMI and PDE approach of \cite{lu1997robustness} for a particular class of systems.


We say that incremental $L^2$ gain is bounded by $\alpha>0$ if the following condition, stronger than \eqref{eq:L2}, is satisfied \cite{desoer1975feedback}:
\begin{equation}\label{eq:incL2}
\int_0^T|y_1-y_2|^2dt \le \alpha^2\int_0^T|w_1-w_2|^2dt
\end{equation}
for all $T>0$ and all pairs of solutions $w_1\rightarrow y_1$ and $w_2\rightarrow y_2$ with $w_1,w_2\in L^2[0,T]$ and equal initial conditions. This can be considered as condition of Lipschitz continuity of the nonlinear operator $w\mapsto y$. The main difference is that $L^2$ gain has a particular ``favored'' solution (e.g. the origin) about which gains are computed, while incremental $L^2$ gain must hold about all solutions.

There are well-known equivalences between \eqref{eq:incL2} and a bound on the gain of the Gateaux derivative of the operator $w\mapsto y$ when such a derivative exists. This has previously utilized to give a generalization of the gap metric in \cite{georgiou1993differential}, analysis of gain scheduling in \cite{fromion2003theoretical}, nonlinear $H^\infty$ control in \cite{fromion2001weighted}, and linear approximation of nonlinear systems in \cite{schweickhardt2009system}.

For nonlinear state-space systems, the Gateaux derivative, when it exists, is given by the differential dynamics (see, e.g, \cite{georgiou1993differential}, \cite{fromion2003theoretical}):
\begin{align}\label{eq:diffdynL2}
\dot\delta_x& = A(x,w,t)\delta_x+B_w(x,w,t)\delta_w \\ \delta_y &= C(x,w,t)\delta_x + D(x,w,t)\delta_w\label{eq:diffdynL22}
\end{align}
where $A = \pder[f]{x}, B = \pder[f]{w}, C = \pder[g]{x}, D = \pder[g]{w}$. And we can consider the {\em differential $L^2$ gain bound} of $\alpha$:
\begin{equation}\label{eq:diffL2}
\int_0^T|\delta_y|^2dt \le \alpha^2\int_0^T|\delta_w|^2dt
\end{equation}
for all $T$ and all solutions of \eqref{eq:sysw}, \eqref{eq:diffdynL2} with $w, \delta_w\in L^2[0,T]$.


Using $\delta_x'M(x,t)\delta_x$ as a differential storage function, it is straightforward to show that satisfaction of the following pointwise LMI:
\begin{equation} \notag
\begin{bmatrix}
\dot M+A'M + MA +C'C& MB_w + C'D\\B_w'M + D'C&D'D-\alpha^2 I\end{bmatrix} \le 0
\end{equation}
guarantees that the system is differentially $L^2$ bounded by $\alpha$ and, therefore, incrementally $L^2$ bounded by $\alpha$. 

The search for a feedback controller guaranteeing incremental or differential $L^2$ gain has been considered in, e.g., \cite{fromion2001weighted} and \cite{fromion2003theoretical}. The search for a {\em differential} gain $\delta_u = K(x)\delta_x$ is straightforward, but it is difficult to ensure that it is integrable, i.e the existence of a feedback $u = k(x)$ such that $\pder[k]{x} = K(x)$.

\subsection{Robust Control Design using CCM}
In this section, we show that the CCM method can be used to chart a middle path between guaranteeing $L^2$ gain for a {\em particular} solution, e.g. the origin, and guaranteeing differential/incremental $L^2$ gain. Consider systems of the form:
\begin{equation}\label{eq:sys_uw}
\dot x = f(x,w,t) +B(x,t)u, \ \ y = g(x,u,t).
\end{equation}
The objective is to use the feedback control $u$ to reduce the effect of disturbance $w$ on output $y$.
The differential dynamics of the system have the form
\begin{align}\label{eq:diffdyn_uw}
\dot \delta_x &= A(x,u,w,t)\delta_x+B(x,t)\delta_u+B_w(x,w,t)\delta_w,\\ \delta_y &= C(x,t)\delta_x+D(x,t)\delta_u\label{eq:diffdyn_uw2}
\end{align}
where $A, B_w, C, D$ are the Jacobians of $\dot x$ and $y$ with respect to $x, w, x$ and $u$ respectively.

\begin{theorem}
Consider the system \eqref{eq:sys_uw} and associated differential dynamics \eqref{eq:diffdyn_uw},  \eqref{eq:diffdyn_uw2}.
Suppose there exists a $W(x,t)=W(x,t)'>0, Y(x,t,u)$, with the latter affine in $u$, that satisfies
\begin{equation}\label{eq:WcondHinf}
\begin{bmatrix}
\mathcal W & WC'+Y'D'\\CW+DY & I\end{bmatrix}\ge 0,
\end{equation}
for all $x,u,w,t$, with
\[
\mathcal W = \dot W -WA' -AW-BY-Y'B- \frac{1}{\alpha^2}B_wB_w',
\]
then for any target solution $u^\star, x^\star, w^\star$ there exists a state-feedback controller for which 
\begin{equation}
\int_0^T|y-y^\star|^2dt \le \alpha^2\int_0^T|w-w^\star|^2dt,
\end{equation}
for any external input $w$ and any $T>0$.
\end{theorem}

\begin{proof}
Set $M=W^{-1}, K = YM$ and apply Schur complement, then \eqref{eq:WcondHinf} is equivalent to
\begin{align}
\delta_x'(\dot M +A'M+MA-MBK-K'B'M)\delta_x& \\+2\delta_x'MB_w \delta_w -\alpha^2|\delta_w|^2+|\delta_y|^2&\le 0\label{eq:MdotHinf}
\end{align}
for all $x,u,w,t$ and all $\delta_x, \delta_w$. With the differential control law $\delta_u=K(x,u,t)\delta_x$, and $\langle \delta_x, \delta_x\rangle_x = \delta_x'M(x)\delta_x$ we have
\begin{equation}\label{eq:L2decrease}
\frac{d}{dt}\langle \delta_x, \delta_x\rangle_x\le \alpha^2|\delta_w|^2 -|\delta_y|^2.
\end{equation}

Consider the family of minimal geodesics $\gamma(s,t)$ connecting $x$ to $x^\star$, with $\gamma(0,t) = x^\star(t)$, $\gamma(1,t) = x(t)$ and parameterize the disturbance as $\bar w(s,t) = (1-s)w^\star(t)+sw(t)$ so $\delta_w=\pder[\bar w]{s} = w-w^\star$ for all $s$. Then the path-integral control law $u(s,t)$ \eqref{eq:control_int}  satisfies \eqref{eq:L2decrease} for any $x = \gamma(s,t)$ and $\delta = \gamma_s(s,t)$. Now, consider the energy function:
\[
e(x(t), x^\star(t)) = \int_0^1 \langle \gamma_s(s,t), \gamma_s(s,t)\rangle_{\gamma(s,t)}ds
\]
then setting $e(t)=e(x(t), x^\star(t))$ we have
\begin{align}
e(T) -e(0) =& \int_0^T \dot e(x(t),x^\star(t)) dt \le \int_0^T \int_0^1\frac{d}{dt}\langle \gamma_s, \gamma_s\rangle_{\gamma}dsdt \notag\\
\le &  \int_0^T \int_0^1-|\delta_y|^2+\alpha^2|\delta_w|^2 dsdt \notag\\
=&\int_0^T \int_0^1-|\delta_y|^2dsdt + \int_0^T\alpha^2|w-w^\star|^2 dt \notag
\end{align}
By assumption of equal initial conditions $e(0) = 0$, and for any $T$, $e(T)\ge 0$, so we have
\[
\int_0^T \int_0^1|\delta_y|^2dsdt \le \int_0^T\alpha^2|w-w^\star|^2 dt
\]
any by applying the $L^1$ Cauchy-Schwarz inequality to the left-hand side integration over $s$, we have
\[
\int_0^T |y-y^\star|^2dt \le \int_0^T\alpha^2|w-w^\star|^2 dt.
\]
This completes the proof of the theorem.
\end{proof}

Note that in general the differential control gain is not integrable, so this control strategy does {\em not} necessarily ensure the incremental gain of the system is bounded between {\em any} pair of solutions of the closed-loop system, so it is weaker than incremental/differential $L^2$ gain. However, it is stronger than $L^2$ gain about the origin, since it implies that the {\em same} metric $M(x,t)$ can be used to bound the gain from {\em any} specified target trajectory.

\section{Observer Design}\label{sec:observer}

It is well-known that the problems of control design and observer design for linear systems have a very attractive ``duality'' (see, e.g., \cite{hespanha2009linear}). In this section, we show that control contraction metrics admit a related duality with nonlinear observer designs using Riemannian metrics, as appearing for example in \cite{tsinias1989observer}, \cite{tsinias1990further}, \cite{sanfelice2012convergence}. We also provide a simple construction of a reduced order observer which fits within the framework of \cite{karagiannis2008invariant}.

Now we consider the problem of state observers, for the non-autonomous system
\begin{equation}\label{eq:obs_sys}
\dot x = f(x,t), \ 
y=c(x,t)
\end{equation}
with differential dynamics
\[
\dot \delta_x = A(x,t)\delta_x, \ \delta_y=C(x,t)\delta_x(t).
\]
with $A(x,t) = \pder[f]{x}$ and $C(x,t) = \pder[c]{x}$.

Following the linear theory, if we simply replace $A(x,t)$ with $A(x,t)'$ and $B(x,t)$ with $C(t)'$ in \eqref{eq:CCM_rhoform} we obtain the condition
\begin{equation}\label{eq:OCM}
\dot M+A'M+MA-\rho C'C< 0,
\end{equation}
where $\alpha_1 I\le M(x,t) \le \alpha_2 I$ for all $x,t$. By Finsler's theorem, \eqref{eq:OCM} is equivalent to the statement that
\begin{equation}\label{eq:Obs_Finsler}
C(x,t)\delta_x(t)= 0 \Longrightarrow \frac{d}{dt}[\delta_x'M(x,t)\delta_x]<0.
\end{equation}
That is, state differentials tangent to the set $\{x:y = c(x,t)\}$ are contracting. Exponential contraction with rate $\lambda$ is given be replacing the $<0$ with $\le -2\lambda M$.

When $M$ was constant or dependent on time, but independent of $x$, conditions of this form were studied in \cite{tsinias1989observer} and \cite{tsinias1990further}. It was shown that a semi-globally stable observer can be constructed and that this can be extended to global stability if a global bound on $\pder[f]{x}$ is satisfied. State-dependent metrics satisfying \eqref{eq:OCM} were studied  \cite{sanfelice2012convergence} in the context of necessary conditions for an observer to make a Riemannian error metric contract, and a semi-globally convergent observer design was  proposed for single-output systems.

\subsection{A Globally Convergent Reduced-Order Observer}

We first consider systems with a linear output map:
\[
\dot x = f(x,t), \quad y = Cx,
\]
the signals
$u(t)$  and $y(t)$ are measured, and the objective is to estimate $x$. In fact, will assume that the state basis is chosen so that
$
C = [I_p \quad 0_{p\times (n-p)}].
$
If the state is not initially in this form, this can easily be achieved by the change of coordinates
\[
x\mapsto \begin{bmatrix}C\\R\end{bmatrix}x
\]
where the rows of $R$ are a basis for $\ker C$. 

Suppose there is a metric $M(x,t)$ and scalar function $\rho(x,t)$ satisfying \eqref{eq:OCM} such that $M$ decomposes as
\[
M = \begin{bmatrix}M_{11}(x,t) & M_{21}'(t)\\ M_{21}(t) & M_{22}(t)\end{bmatrix},
\]
i.e. only the upper-left block depends on $x$. Since \eqref{eq:OCM} is equivalent to \eqref{eq:Obs_Finsler}, and differentials satisfying $C\delta_x=0$ are of the form $[0_{1\times p},\quad \delta_2']'$, the lower-right block of the condition \eqref{eq:OCM}:
\begin{equation}\label{eq:M22cond}
\dot M_{22} + M_{21}A_{12} + A_{12}'M_{21}' + M_{22}A_{22} + A_{22}'M_{22} \le - 2\lambda M_{22}.
\end{equation}
is an equivalent condition for exponential convergence.

Define $C^+$ and $P$ as pseudo-inverse and projection operators defined with respect to the inner-product space $\langle a, b\rangle = a'M(x,t)b$:
\begin{align}
C^+=\begin{bmatrix}I_p\\ -M_{22}^{-1} M_{21}\end{bmatrix},\\
P = [M_{22}^{-1} M_{21} \quad I_{n-p}].
\end{align}

\begin{theorem} Suppose there is a solution  to \eqref{eq:M22cond} with $M_{22}(t)>\alpha I$, $\alpha>0$, then the following reduced-order observer is globally exponentially stable with rate $\lambda$:
\begin{align}
\dot {\hat w} = \dot P\hat x + Pf(\hat x,t), \quad \hat x = C^+y + \begin{bmatrix}0\\\hat w\end{bmatrix}.
\end{align}
\end{theorem}

\begin{proof}
We show that the $\hat w$ system is contracting exponentially for any input $y(t), u(t)$, and that $\hat x = x$ is a particular solution, which proves that all solutions have $\hat x \rightarrow x$ exponentially.

The differential dynamics of $w$ are:
\begin{align}
\dot\delta_w =& \frac{d}{dt} ([M_{22}^{-1} M_{21} \quad I_{n-p}]) \begin{bmatrix}0\\ I_{n-p}\end{bmatrix} \delta_w \notag\\ &+ [M_{22}^{-1} M_{21} \quad I_{n-p}]A(x,t)\begin{bmatrix}0\\ I_{n-p}\end{bmatrix} \delta_w \notag
\end{align}
where the first term evaluates to zero.

To show that the $\hat w$ system is contracting we use the metric $\delta_w'M_{22}\delta_w$.  Differentiating with respect to time gives
\[
\delta_w'(\dot M_{22} + 2M_{21}A_{12}(x) + 2M_{22}A_{22}(x))\delta_w \le -2\lambda \delta_w'M_{22}\delta_w.
\]
where the inequality comes from \eqref{eq:M22cond}.

We also show that the true state is a particular solution. Indeed, suppose $\hat w = w$ and $\hat x = x$, then $f(\hat x,t) = f(x,t) = \dot x$ and
\begin{align}
\dot {\hat x} &= \dot C^+y+ C^+\dot y +\begin{bmatrix}0\\ \dot P\end{bmatrix}\begin{bmatrix}y\\ w\end{bmatrix}+ \begin{bmatrix}0\\ P\end{bmatrix}\begin{bmatrix}\dot y\\ \dot w\end{bmatrix} \notag\\
& = \begin{bmatrix}I_p\\ -M_{22}^{-1} M_{21}\end{bmatrix}\dot y + \begin{bmatrix}0 & 0 \\ M_{22}^{-1} M_{21} & I_{n-p}\end{bmatrix}\begin{bmatrix}\dot y\\ \dot w\end{bmatrix}=\begin{bmatrix}\dot y\\ \dot w\end{bmatrix} = \dot x.
\end{align}
Here we have used the fact that
\[
\dot C^+y = -\begin{bmatrix}0\\ \dot P\end{bmatrix}\begin{bmatrix}y\\ w\end{bmatrix}.
\]
Since all solutions of $\hat w$ converge, and the true state is a particular solution, all initial conditions converge to the true state.
\end{proof}

\subsection{Nonlinear Output Maps and State-Dependent Metrics}

Let us briefly consider a way to extend the above results to problems with nonlinear output maps. For brevity we consider systems \eqref{eq:obs_sys} independent of time. The approach we suggest is to search for a nonlinear change of coordinates that puts the system in the form in the previous subsection. We propose a method to convexify the joint search for this change of coordinates and a contraction metric, adapting recent research in nonlinear system identification \cite{tobenkin2010convex}.

We wish to find  $r:\RR^n\rightarrow \RR^{n-p}$ so that the mapping
\[
z=\phi(x) = \begin{bmatrix}c(x)\\r(x)\end{bmatrix}
\]
is a global diffeomorphism.

\begin{theorem}
Suppose there exists a $r(x)$ defining $\phi(x)$ as above, which satisfies the following conditions, which are jointly convex in $\rho(x), r(x)$ and matrix variable $Q = Q'>0$:
\begin{align}
(\Phi+F)'Q^{-1}(\Phi+F) + Q & \notag \\
-(\Phi-F)-(\Phi-F)' -2\rho C'C& \le -4\lambda I,\label{eq:convexrel}\\
\Phi(x)+\Phi(x)'&\ge 2\mu I,\label{eq:Ebnd}
\end{align}
where $\Phi(x) = \pder[\phi]{x}$ and $F(x) = \pder{x}(\Phi(x)f(x))$. Then there exists a reduced order observer for the system, using the construction in the previous subsection.
\end{theorem}
\begin{proof}
We first show that $\phi$ is a diffeomorphism. The condition \eqref{eq:Ebnd} implies the autonomous dynamical system
$
\dot x = -\phi(x)+z
$
is contracting, and therefore has a unique equilibrium $z=\phi(x)$, hence the nonlinear mapping $\phi$ has well-defined inverse. From \eqref{eq:Ebnd} it is also clear that the Jacobian of $\phi (\cdot)$ is non-singular for every $x$, so the inverse of $\phi$ is everywhere differentiable. 

Secondly, we show that \eqref{eq:convexrel} implies contraction of the observer error. Take the metric
\[
\delta_zQ^{-1}\delta_z = \delta_x\Phi(x)'Q^{-1}\Phi(x)\delta_x=:\langle \delta_x, \delta_x\rangle_x.
\]
Then its derivative is
\[
\frac{d}{dt}\langle \delta_x, \delta_x\rangle_x = 2\delta_x\Phi(x)'Q^{-1}F(x)\delta_x
\]
with $F$ defined as above. Now, taking the symmetric part of the matrix the polarisation identity:
\begin{align}
\Phi'Q^{-1}F + \Phi'Q^{-1}F =& \frac{1}{2}(\Phi+F)'Q^{-1}(\Phi+F)\notag\\ &-\frac{1}{2}(\Phi-F)'Q^{-1}(\Phi-F)\notag
\end{align}
where the first term is convex in $e, Q$ and the second term is concave. Now by expanding the relation
\[
\left(I-Q^{-1}(\Phi-F)\right)'Q\left(I-Q^{-1}(\Phi-F)\right)\ge 0
\]
we get
\[
-(\Phi-F)'Q^{-1}(\Phi-F)\le Q-(\Phi-F)-(\Phi-F)'
\]
so the right-hand side is a linear (and hence convex) upper bound for the concave part above. Therefore \eqref{eq:convexrel} implies that
\[
\frac{d}{dt}\langle \delta_x, \delta_x\rangle_x\le \rho(x)\delta_x'C'C\delta_x-2\lambda |\delta_x|^2
\]
which is the observer contraction condition.
\end{proof}
Current research considers extensions of these ideas to adaptive control and estimation with nonlinear parametrizations.

\section{Computational Example}

The Moore-Greitzer model, a simplified model of surge-stall dynamics of a jet engine  \cite{moore1986theory}, has motivated substantial development in nonlinear control design (see, e.g., \cite{moore1986theory},  \cite{krstic1995nonlinear}, and references therein). 

A model of surge-stall dynamics was derived in \cite{moore1986theory} based on a Galerkin projection of the PDE on to a Fourier basis. The following reduced model of the surge dynamics is frequently studied:
\[
\begin{bmatrix}\dot\psi \\ \dot\phi \end{bmatrix}  = \begin{bmatrix} \phi + u \\-\psi -\frac{3}{2}\phi^2 - \frac{1}{2}\phi^3+w\end{bmatrix}.
\]
with $u$ as the input and a sensor on $\psi$ only. Here $\phi$ is a measure of mass flow through the compressor, and $\psi$ is a measure of the pressure rise in the compressor, under a change of coordinates, see \cite[p. 68]{krstic1995nonlinear}. 

We first computed a robust controller for using the method in Section \ref{sec:robust}, with $C=[0, 1]$, $D=0.1$. It follows from the unbounded growth of $A$ that \eqref{eq:WcondHinf} cannot be satisfied across all $x$, however it can be satisfied on any compact set. This can be accomplished using Lagrange multipliers and the sum-of-squares relaxation. We used Yalmip \cite{YALMIP}, \cite{Lofberg09} and Mosek to search for a constant metric $W$ and gain $Y(\phi)$, quadratic in $\phi$, satisfying the following convex constraints:
\begin{align}
v'\begin{bmatrix}
\mathcal W & WC'+Y'D'\\CW+DY & I\end{bmatrix}v - \tau(\phi,v)(r^2-\phi^2)&\ge 0,\notag\\
AW+WA'+BY+Y'B'+\lambda W &\le 0,\notag \\
\tau(\phi,v)\ge 0, \quad\quad
W&\ge 0.\notag
\end{align}
where the Lagrange multiplier $\tau$, quadratic in $\phi$ and $v$, guarantees the gain bound on the set $|\phi|<r$. The second constraint ensures that the closed-loop system is globally exponentially stable. The results are shown in Table \ref{tab1}. 

\begin{table}
\begin{center}
\caption{Best achieved gain bounds for different radii of $\phi$.}\label{tab1}
    \begin{tabular}{|l|c| c| c| c|c|}
	\hline\hline Radius $r$ &1& 5 &10&20&40 \\ 
	\hline $\alpha$ &0.49&1&1.74&3.16&6.1 \\ \hline
    \end{tabular}
\end{center}
\end{table}

We also computed a reduced-order observer using the methods in Section \ref{sec:observer}, assuming that only pressure $\psi$ is measureble, i.e. $y = \phi$.

Again, this design problem could be solved using convex optimization, but for the Moore Greitzer compressor model the computations are actually trivial. Indeed,   
$A_{22}(x) = -3\phi -\frac{3}{2}\phi^2$ and $A_{12} = 1$ so the observer contraction condition \eqref{eq:M22cond} is simply the existence of numbers $M_{22}>0$ and $M_{21}\in\RR$ satisfying
\[
M_{21}-M_{22}(3\phi +3/2\phi^2-\lambda)\le 0.
\]
Since the left-hand-side is concave in $\phi$ with a maximum at $\phi=-1$, the condition is satisfied for, e.g. $M_{22}=1$ and $\lambda = 0.5$, we can take
$
M_{21} = -1.5-\lambda=-2.
$
Note that the observer construction only depends on the ratio of $M_{21}/M_{22}$.

Simulation results can be seen in Figure \ref{fig:MeasNoise}, where a Gaussian white noise with standard deviation 0.2 was added to $y$. It can be seen that despite the noise, the estimate of the unobserved state $\phi$ converges rapidly and is not greatly perturbed by the noise. 

\begin{figure}
\begin{center}
\includegraphics[width=0.9\columnwidth]{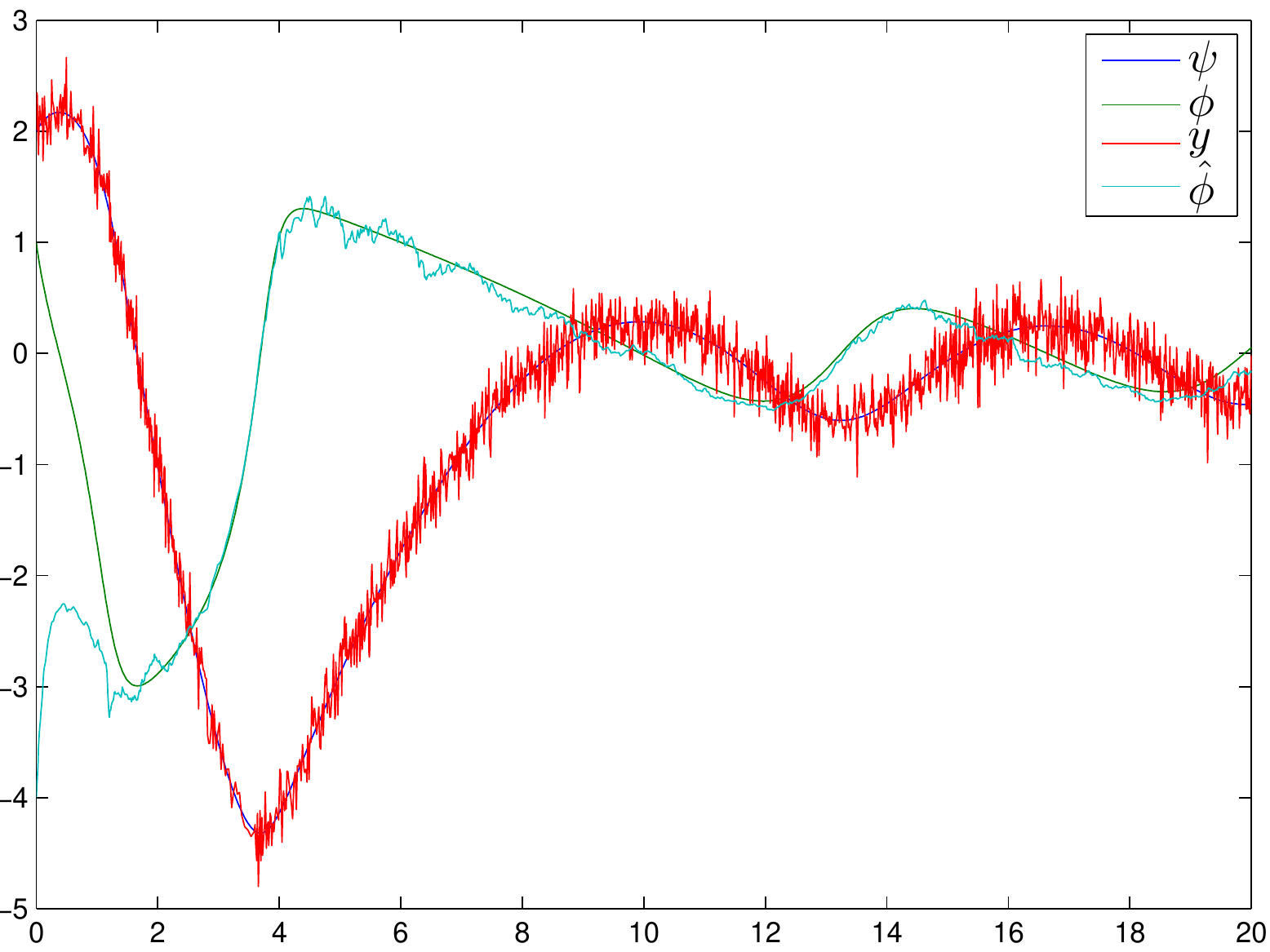}
\caption{Reduced order observer for the Moore-Greitzer system: true states and estimates vs time (s).}
\label{fig:MeasNoise}
\end{center}
\end{figure}

\bibliographystyle{IEEEtran}
\bibliography{elib}

\end{document}